\documentclass[11pt]{amsart}
\usepackage{amssymb,amsmath,amsthm}
\usepackage{tikz-cd}

\newtheorem{prop}{Proposition}
\newtheorem{theorem}[prop]{Theorem}
\newtheorem{cor}[prop]{Corollary}

 \newcommand{\Q}{\ensuremath{\mathbb{Q}}}
 
 \newcommand{\C}{\ensuremath{\mathbb{C}}}

\newcommand{\Gal}{{\rm Gal}}

\newcommand{\Ker}{{\rm Ker \,}}

\newcommand{\Reg}{{\rm Reg }}

\begin{document}
\title{Are number fields determined by Artin $L$-functions?}
\author{J\"urgen Kl\"uners}
\address{Mathematisches Institut der Universit\"at Paderborn,
Warburger Str. 100, 33098 Paderborn, Germany}
\email{klueners@math.uni-paderborn.de}
\author{Florin Nicolae}
\address{``Simion Stoilow'' Institute of Mathematics of the
Romanian Academy, P.O.BOX 1-764,RO-014700 Bucharest, Romania}
\email{florin.nicolae@imar.ro}

\begin{abstract}
Let $k$ be a number field, $K/k$ a finite Galois extension with Galois group $G$, $\chi$ a faithful character of $G$. 
We prove that the Artin L-function $L(s,\chi,K/k)$ determines the Galois closure of $K$ over $\Q$. In the special case
$k=\Q$ it also determines the character $\chi$.


\end{abstract}
\maketitle

\section{Introduction}
\thispagestyle{empty}

Let $k$ be a number field, $K/k$ a finite Galois extension
with Galois group $G$, $\chi$ a faithful character of $G$.  In Theorem
\ref{thm6} we prove that the Artin $L$-function $L(s,\chi,K/k)$
determines the Galois closure $\tilde{K}$ of $K$ over $\Q$. In the
special case $k=\Q$ we prove in Theorem \ref{thm5} that the Artin
$L$-function determines $K$ and the (faithful) character $\chi$. We give
examples that in the case $k\ne \Q$ we cannot expect more, especially there exist 
non-isomorphic arithmetically equivalent fields which cannot be distinguished
by Artin $L$-functions.

The restriction to faithful characters is natural: let $K/k$ be a
finite normal extension with $\Gal(K/k)=G$, and let $\chi$ be a
character of $G$ with $\Ker(\chi)\neq \{1\}$. Let $F$ be the fixed
field of $\Ker(\chi)$, $H:=G/\Ker(\chi)$ the Galois group of $F/k$,\;
$\varphi:H\to \C$, $\varphi(\sigma\Ker(\chi)):=\chi(\sigma)$ for
$\sigma\in G$. We have that
$$L(s,\chi,K/k)=L(s,\varphi,F/k),$$  
and $\varphi$ is faithful.

As particular cases we obtain that the Dedekind zeta function
of a number field determines its normal closure (\cite{Pe}, Theorem 1,
p. 345) and that a Galois number field is determined by any Artin
$L$-function corresponding to a character which contains all irreducible
characters of the Galois group, the result of \cite{Ni2}.

\section{Properties of Artin $L$-Functions}

We do not give the definition of Artin $L$-functions, but
we recall some fundamental properties of Artin $L$-functions needed in  the sequel. Note
that Artin $L$-functions are generalizations of Dedekind zeta functions $\zeta_K$ via
$$L(s,1,K/K) = \zeta_K(s),$$
where $K$ is a number field and $1$ is the trivial character of the trivial group $\Gal(K/K)$. We get further possibilities
to write a Dedekind zeta function as a Artin $L$-function by using Propositions \ref{prop1} and \ref{prop2}.

\begin{prop}\label{prop1}
Let $k$ be a number field, $K/k$ a finite Galois extension with Galois
group $G$, $\chi$ a character of $G$. Let $N$ be a finite Galois
extension of $k$ which contains $K$, $U:=\Gal(N/k)$,
$V:=\Gal(N/K)$. We identify the groups $G$ and
$U/V$. Let $$\tilde{\chi}:U\to\mathbb C,\;\tilde{\chi}(\sigma):=\chi(\sigma V).$$ Then we have
$$L(s,\tilde{\chi},N/k)= L(s,\chi,K/k).$$
\end{prop}

\begin{proof} This follows straightforward from the definition of $L$-functions: \cite{Ar}, p. 297, formula (8).
\end{proof}

\begin{prop}\label{prop2}
 Let $k$ be a number field, $K/k$ a finite Galois extension with Galois group $G$. Let $k\subseteq F\subseteq K$ be an intermediate 
 field, $H:=\Gal(K/F)$, $\chi$ a character of $H$, and $\chi^G$ be the induced character on $G$. Then we have
 $$L(s,\chi^G,K/k)=L(s,\chi, K/F).$$
\end{prop}

\begin{proof} This is a deep property proved by Emil Artin: \cite{Ar}, p. 297, formula (9).
\end{proof}

\begin{prop} \label{prop3}
 Let $K/\Q$ be a finite Galois extension with Galois group $G$, $\chi$ and $\varphi$ characters of $G$. If 
 $$L(s,\chi, K/\Q)=L(s,\varphi, K/\Q)$$ then
 $$\chi=\varphi.$$
\end{prop}

\begin{proof} This follows from the definition of Artin $L$-functions and Tschebotarev's  density theorem. Alternately,  
it was proved in \cite{Ni1}, p. 179, Theorem 1 that if $\chi\neq\varphi$ then the functions $L(s,\chi, K/\Q)$ and
$L(s,\varphi, K/\Q)$ 
are linearly independent over $\C$. 
\end{proof}
We remark that Proposition \ref{prop3} is not true for Artin $L$-functions in Galois extensions $K/k$ with base $k\neq \Q$. E.g. let $K/\Q$
be a normal $S_3$--extension and denote by $k$ the unique quadratic subfield of $K/\Q$. Then $\Gal(K/k) = C_3$, the cyclic group of order 3.
Denote by $\chi_2$ and $\chi_3$ the two non-trivial characters of $C_3$. Then by Proposition \ref{prop2} we get for $j=2,3$:
$$L(s,\chi_j^{S_3},K/\Q)=L(s,\chi_j, K/k).$$
Using $\chi_2^{S_3}=\chi_3^{S_3}$ we get that $L(s,\chi_2,K/k) = L(s,\chi_3,K/k)$, but $\chi_2\ne\chi_3$.

\section{Results}

\begin{prop}\label{prop4}
Let $k$ be a number field, $K/k$ a finite Galois extension with Galois group $G$, and $\chi$ be a faithful character of $G$. Let $\tilde{K}$ 
be the Galois closure of $K$ over $\Q$, $\tilde{G}$ the Galois group of $\tilde{K}/\Q$, $U:=\Gal(\tilde{K}/k)$,  
and $V:=\Gal(\tilde{K}/K)$. We identify the groups $G$ and $U/V$. Let $$\tilde{\chi}:U\to\mathbb C,\;
\tilde{\chi}(\sigma):=\chi(\sigma V), $$ and let $\tilde{\chi}^{\tilde{G}}$ be the induced character of $\tilde{\chi}$ on $\tilde{G}$. 
Then $\tilde{\chi}^{\tilde{G}}$ is faithful, and 
$$L(s,\chi,K/k)=L(s, \tilde{\chi}^{\tilde{G}},\tilde{K}/\Q). $$
\end{prop}
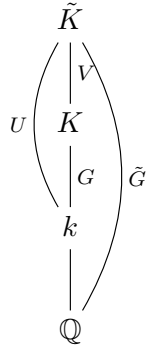
\begin{center}
\begin{tikzcd}
    \tilde{K} \arrow[-]{d}{V} \arrow[bend right,-]{dd}[swap]{U}  \arrow[bend left,-]{ddd}{\tilde G}\\
  K \arrow[-]{d}{G} \\
  k \arrow[-]{d} \\
  \Q\\
\end{tikzcd}
\end{center}

\begin{proof}
We have $$\Ker(\tilde{\chi})=V,$$ since $\chi$ is faithful.
We have $$\bigcap_{\sigma\in \tilde{G}}\sigma V\sigma^{-1}=1 ,$$ 
since $\tilde{K}$ is the Galois closure of $K$ over $\Q$.
We have 
$$\Ker(\tilde{\chi}^{\tilde{G}})=\bigcap_{\sigma\in \tilde{G}}\sigma\Ker(\tilde{\chi})\sigma^{-1}=
\bigcap_{\sigma\in \tilde{G}}\sigma V\sigma^{-1}=1,$$
hence $\tilde{\chi}^{\tilde{G}}$ is faithful. We have
$$L(s,\chi,K/k)=L(s,\tilde{\chi},\tilde{K}/k)=L(s, \tilde{\chi}^{\tilde{G}},\tilde{K}/\Q),$$
by Propositions \ref{prop1} and \ref{prop2}.
\end{proof}

We prove that a Galois number field is determined by any Artin $L$-function corresponding to a faithful character of the Galois group.

\begin{theorem}\label{thm5}
Let $K_1/\Q$, $K_2/\Q$ be finite Galois extensions. For $j=1,2$ let $G_j$ be the 
Galois group of $K_j/\Q$, $\chi_j$ a faithful character 
of $G_j$. If 
$$L(s,\chi_1,K_1/\Q)=L(s,\chi_2,K_2/\Q) $$
then
$$K_1=K_2 \mbox{ and } \chi_1=\chi_2.$$
\end{theorem}

\begin{proof}
Let $N:=K_1\cdot K_2$. For $j=1,2$ let 
$$\tilde{\chi_j}:\Gal(N/\Q)\to\mathbb C,\;
\tilde{\chi_j}(\sigma):=\chi_j(\sigma\Gal(N/K_j)),$$
where we identify $G_j$ with the factor group 
$\Gal(N/\Q)/\Gal(N/K_j)$. Then $\tilde{\chi_j}$ is a
character of $\Gal(N/\Q)$, and
$$\Ker(\tilde{\chi_j})=\Gal(N/K_j),$$ since $\chi_j$ is faithful.  
We have  
$$ L(s,\chi_j,K_j/\Q)= L(s,\tilde{\chi_j},N/\Q),$$
by Proposition \ref{prop1}.
If 
$L(s,\chi_1,K_1/\Q)=L(s,\chi_2,K_2/\Q)$ then
$$ L(s,\tilde{\chi_1},N/\Q)= L(s,\tilde{\chi_2},N/\Q),$$
hence, by Proposition \ref{prop3}, $$\tilde{\chi_1}= \tilde{\chi_2}.$$
This implies
$$\Gal(N/K_1)=\Ker(\tilde{\chi_1})= \Ker(\tilde{\chi_2})=\Gal(N/K_2),$$
hence
$$K_1=K_2$$
and, by Proposition \ref{prop3},
$$\chi_1=\chi_2.$$
\end{proof}

In the final example of this paper we show that we cannot expect a similar result to Theorem \ref{thm5} for normal extensions $K_j/k$ for
$k \ne \Q$ and $j=1,2$. In the next theorem we give a version for relative extensions. 

\begin{theorem}\label{thm6}
Let $k_1$ and $k_2$ be number fields. For $j=1,2$ let $K_j/k_j$ be a finite Galois extension with the Galois group $G_j$, 
$\tilde{K_j}$ the normal closure of $K_j$ over $\Q $, $\chi_j$ a faithful character of $G_j$, $U_j:=\Gal(\tilde{K_j}/k_j)$, 
$V_j:=\Gal(\tilde{K_j}/K_j)$. We identify the groups $G_j$ and $U_j/V_j$. Let 
$$\tilde{\chi_j}:U_j\to\mathbb C,\;
\tilde{\chi_j}(\sigma):=\chi_j(\sigma V_j).$$ 
Let $\tilde{G_j}:= \Gal(\tilde{K_j}/ \Q)$, and let $\tilde{\chi_j}^{\tilde{G_j}}$ be the induced character of $\tilde{\chi_j}$ on $\tilde{G_j}$.  
If $$L(s,\chi_1,K_1/k_1)= L(s,\chi_2,K_2/k_2)$$ 
then $$\tilde{K_1}=\tilde{K_2}, \; \tilde{G_1} = \tilde{G_2}$$ and
$$\tilde{\chi_1}^{\tilde{G_1}}= \tilde{\chi_2}^{\tilde{G_2}}.$$
\end{theorem}
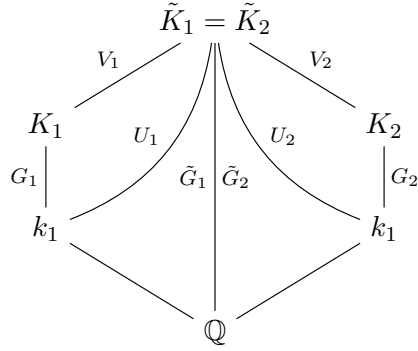
\begin{center}
\begin{tikzcd}
  & \tilde{K}_1 = \tilde{K}_2 \arrow[-]{dl}[swap]{V_1}  \arrow[-]{dr}{V_2} \arrow[bend left,-]{ddl}[swap]{U_1} \arrow[bend right,-]{ddr}{U_2} \arrow[-]{ddd}{\tilde{G}_2}[swap]{\tilde{G}_1} \\
K_1\arrow[-]{d}[swap]{G_1} &  &  K_2\arrow[-]{d}{G_2}\\
k_1 \arrow[-]{dr}&  &  k_1\arrow[-]{dl}\\
   & \Q & \\
\end{tikzcd}
\end{center}

\begin{proof}

For $j=1,2$ the character $\tilde{\chi_j}^{\tilde{G_j}}$ is faithful and we have 
$$L(s,\chi_j, K_j/k_j)=L(s,\tilde{\chi_j}^{\tilde{G_j}},\tilde{K_j}/\Q), $$
by Proposition \ref{prop4}. If $$L(s,\chi_1,K_1/k_1)= L(s,\chi_2,K_2/k_2)$$ then 
$$L(s,\tilde{\chi_1}^{\tilde{G_1}},\tilde{K_1}/\Q)= L(s,\tilde{\chi_2}^{\tilde{G_2}},\tilde{K_2}/\Q),$$
hence $$\tilde{K_1}=\tilde{K_2}, \; \tilde{G_1} = \tilde{G_2}$$ and
$$\tilde{\chi_1}^{\tilde{G_1}}= \tilde{\chi_2}^{\tilde{G_2}},$$
by Theorem \ref{thm5}.
\end{proof}    

We obtain now the well-known result (\cite{Pe}, Theorem 1, p. 345) that the zeta function of an algebraic number field determines its 
normal closure.

\begin{cor}
Let $K_1$ and $K_2$ be number fields. For $j=1,2$ let $\tilde{K_j}$ be the normal closure of $K_j$ over $\Q $, $V_j:=\Gal(\tilde{K_j}/K_j)$, 
$\tilde{G_j}:= \Gal(\tilde{K_j}/ \Q)$, and let $1_{V_j}$ be the trivial character of $V_j$. If 
$$\zeta_{K_1}(s)=\zeta_{K_2}(s) $$
then  
$$\tilde{K_1}=\tilde{K_2},\, \tilde{G_1}= \tilde{G_2}$$
and 
$$1_{V_1}^{\tilde{G_1}}=1_{V_2}^{\tilde{G_2}}.$$
\end{cor}
\begin{proof}
For $j=1,2$ let  $k_j=K_j$, $G_j= \{1\}$, $\chi_j:=1_{G_j}$, $U_j:=V_j$. 
The character $\chi_j $ is faithful, and $$L(s,\chi_j, K_j/k_j)= \zeta_{K_j}(s).$$ We have 
$$\tilde{\chi_j}:U_j\to\mathbb C,\;
\tilde{\chi_j}(\sigma):=\chi_j(\sigma V_j)=\chi_j(1)=1,$$
hence $$ \tilde{\chi_j}=1_{V_j}.$$
If 
$$ \zeta_{K_1}(s)=\zeta_{K_2}(s)$$
then 
$$L(s,\chi_1,K_1/k_1)= L(s,\chi_2,K_2/k_2), $$
hence
$$\tilde{K_1}=\tilde{K_2},\,\tilde{G_1}= \tilde{G_2}  $$
and
$$1_{V_1}^{\tilde{G_1}}=1_{V_2}^{\tilde{G_2}},$$
by Theorem \ref{thm6}.
\end{proof}                     
We finish with an example which shows that all Artin $L$-functions in the Galois closures of the corresponding fields coincide.
Let $k=k_1=k_2:=\Q(\sqrt[4]{3})$, $K_1:=\Q(\sqrt[8]{3})$, $K_2:=\Q(\sqrt[8]{16\cdot 3})$. The extensions $K_1/k$ and $K_2/k$ are Galois with Galois groups 
$G_1$ and $G_2$ of order 2.
The fields $K_1$ and $K_2$ are non-isomorphic and have the same zeta function ( \cite{Pe}, p. 350 ):
$$L(s,\Reg_{G_1}, K_1/k)=\zeta_{K_1}(s)= \zeta_{K_2}(s)=L(s,\Reg_{G_2}, K_2/k).$$ For $j=1,2$ we have 
$$\Reg_{G_j}=1_{G_j}+\chi_j, $$ where $\chi_j$ is the non-trivial irreducible character of $G_j$. We have 
$$L(s,1_{G_1}, K_1/k)=\zeta_k(s)= L(s,1_{G_2}, K_2/k),$$hence  
$$L(s,\chi_1, K_1/k)=L(s,\chi_2, K_2/k).$$ It follows that the Artin $L$-functions of $ K_1/k$ are identical with the Artin $L$-functions of $ K_2/k$. We have 
$$\tilde{K_1}=\tilde{K_2},\,\tilde{G_1}= \tilde{G_2}=G,$$
by Theorem 6, since $\Reg_{G_1}$ and $\Reg_{G_2}$ are faithful. Using the notation of Theorem \ref{thm6}, there is an outer automorphism $\alpha$ of $G$ such that 
$$V_2=\alpha(V_1).$$ The Artin $L$-functions of $\tilde{K_1}/K_1$ are identical with the Artin $L$-functions of $\tilde{K_2}/K_2$. This shows that $K_1$ and $K_2$ cannot be distinguished 
by Artin $L$-functions of $\tilde{K_1}/K_1$ and $\tilde{K_2}/K_2$, nor by Artin $L$-functions of $K_1/k$ and $K_2/k$.

\end{document}